\documentclass[reqno,12pt,a4paper]{amsart}

\usepackage[margin=1in]{geometry}
\usepackage{amssymb,amsfonts,lineno,amsthm,amscd,stmaryrd}
\usepackage[numbers]{natbib}
\usepackage{url}
\usepackage[
	pdftex,
	pdfstartview={XYZ 0 1000 1.0},
	bookmarks=false,
	colorlinks,
	breaklinks=true,
	citecolor=black,
	filecolor=black,
	linkcolor=black,
	urlcolor=black,
	pdfborder={0 0 0}
]{hyperref}
\usepackage{stmaryrd}
\usepackage{mathtools}
\usepackage{nicefrac}
\usepackage{amsfonts}

\newcommand{\begeqna}{\begin{eqnarray}}
\newcommand{\eeqna}{\end{eqnarray}}
\newcommand{\begeqnat}{\begin{eqnarray*}}
\newcommand{\eeqnat}{\end{eqnarray*}}
\newcommand{\be}{\begin{equation}}
\newcommand{\ee}{\end{equation}}
\newcommand{\bet}{\begin{equation*}}
\newcommand{\eet}{\end{equation*}}
\newcommand{\brmk}{\begin{remark} \em}
\newcommand{\ermk}{\end{remark} }

\theoremstyle{plain}
\newtheorem{thm}{Theorem}
\newtheorem{cor}{Corollary}
\newtheorem{lem}[cor]{Lemma}

\theoremstyle{definition}
\newtheorem{definition}[cor]{Definition}

\theoremstyle{remark}
\newtheorem{rmk}[cor]{Remark}

\numberwithin{cor}{section}
\numberwithin{equation}{section}

\newcommand{\EQ}[1]{\eqref{eq:#1}}
\newcommand{\LEM}[1]{Lemma~\ref{lem:#1}}

\newcommand{\THM}[1]{Theorem~\ref{thm:#1}}
\newcommand{\REM}[1]{Remark~\ref{rem:#1}}

\newcounter{hypo}

\DeclareMathOperator{\divg}{div}
\DeclareMathOperator{\sign}{sign}

\newcommand{\R}{\ensuremath{\mathbb{R}}}
\newcommand{\rn}{\R^n}
\newcommand{\Z}{\ensuremath{\mathbb{Z}}}
\newcommand{\N}{\ensuremath{\mathbb{N}}}

\newcommand{\am}{\alpha}

\newcommand{\Dg}{\Delta}

\newcommand{\Gg}{\Gamma}

\newcommand{\iy}{\infty}

\newcommand{\plap}{\Dg_{p}}

\newcommand{\diff}{\mathop{}\!\mathrm{d}}

\newcommand{\ess}{\mathop{}\!\mathrm{ess}}
\newcommand{\loc}{\mathop{}\!\mathrm{loc}}

\newcommand{\ds}{\displaystyle}

\newcommand{\mo}[1]{|{#1}|}

\newcommand{\abs}[1]{\mid #1 \mid}

\def\XXint#1#2#3{{
\setbox0=\hbox{$#1{#2#3}{\int}$}
\vcenter{\hbox{$#2#3$}}\kern-.5\wd0}}

\newcommand{\funcionesc}[5]{\begin{array}{rccl}
#1:&#2&\longrightarrow &#3\\
&#4& \xmapsto{\phantom{\rightarrow }}&#5.\end{array}}

\newcommand{\pvi}[3]{\be\label{#1}
     \begin{cases}#2 & = #3\;\; \mbox{in}\;\;\; (0,\iy)\\
\;\;\;\;\;\;\;\;\;\;\;\;\; \;\; v(0)&=v_{0}>0 \\
\;\;\;\;\;\;\;\;\;\;\;\;\; \;\; v'(0)&=0
\end{cases}
     \ee}

\newcommand{\vk}{v_{k}}
\newcommand{\vkl}{v'_{k}}
\newcommand{\vl}{v'}
\newcommand{\vll}{v''}
\newcommand{\vlp}{(v')^{p-1}}
\newcommand{\vlpl}{((v')^{p-1})'}

\newcommand{\fik}{\phi_{k}}
\newcommand{\mdo}{m_{2}}
\newcommand{\mtre}{m_{3}}

\newcommand{\exist}[4]{H^{-1}\left(\int_{0}^{#1}\left(\frac{t}{#2}\right)^{n-1} F(#3,#4)\diff t\right)}

\begin{document}
\title[Positive solutions of quasi-linear equations with gradient terms]{Existence and nonexistence of positive solutions of quasi-linear elliptic equations with gradient terms}

\author{Dania Gonz\'{a}lez Morales}
\address{Department of Mathematics\\ Pontifical Catholic University \\ Rio de Janeiro Rio de Janeiro, Brazil.}
\email{dania@mat.puc-rio.br}

\begin{abstract}
We study the existence and nonexistence of positive solutions in the whole Euclidean space of coercive quasi-linear  elliptic equations such as
\bet
\plap u = f(u)\pm g(\left|\nabla u\right|)
\eet
where $f\in C([0,\infty))$ and $g\in C^{0,1}([0,\infty)) $ are strictly increasing with $ f(0)=g(0)=0$.

Among other things we obtain generalized  integral conditions of Keller-Osserman type. In the particular case of  plus sign on the right-hand side we obtain that different conditions are needed when $p\geq 2$ or $p\leq 2$, due to  the degeneracy of the operator.

\end{abstract}



\maketitle

\medskip
\medskip



\section{Introduction}

In this paper we study the existence and nonexistence of nonnegative solutions of the coercive quasi-linear elliptic equation with gradient term
\be\label{eq:Pmasmenos}
\plap u = f(u)\pm g(\mo{\nabla u}) \;\; \mbox{in} \; \rn,\tag{$P_{\pm}$}
\ee
where $\plap$ denotes the $p-$Laplacian operator
\be
\label{eq:eqlap}
\plap\cdot:=\divg (\mo{\nabla \cdot}^{p-2}\nabla \cdot), \;\;\; 1<p<\iy,
\ee
 and
\be\label{eq:fgcond}
f\in C([0,\infty)) , g\in C^{0,1}([0,\infty)) \ \mbox{are strictly increasing with} \; f(0)=g(0)=0.
\ee

The $p$-Laplacian operator appears in several contexts, such as  glaciology, radiation of heat, nonlocal diffusion, image and data processing  (for more applications, please consult \cite{elmoataz2015p}, \cite{bueno2012positive}, \cite{mastorakis2009solution} and the references therein).

The main novelty of this work resides in the concurrent resolution of difficulties which arise, on one hand, from the singular or degenerate nature of the operator, and on the other hand, from the general form of the right hand side and the need to understand how the interaction between $f$ and $g$ affects the solvability of the problem. The presence of the gradient term combined with the intrinsic properties of the $p-$Laplacian operator makes our study of this type of problems  different from previous works, to our knowledge.

In the particular case $p=2$ this type of study goes back to  the famous work \cite{lasry1989nonlinear} of Lasry and Lions, who studied explosive solutions defined in a bounded domain. Concerning  unbounded domains,  novel results appeared later  in the paper of Farina and Serrin \cite{farina2011entire}. They studied positive solutions in the whole Euclidean space when the term which depends on the gradient is described by an specific power growth, and the sum is replaced by a product of $f$ and $g$. In a more recent work, Felmer, Quaas and Sirakov \cite{BB} showed results of existence and non existence of \eqref{eq:Pmasmenos}, when the operator is uniformly elliptic, such as the Laplacian, for viscosity solutions in the whole Euclidean space.  They gave a rather precise description of the way that the interaction between the terms $f$ and $g$ in the nonlinearity influences the solvability of these problems.

As far as the general case when $p$ is not necessarily two is concerned, we first quote the pioneering work
\cite{mitidieri1999nonexistence} of Mitidieri and Pohozaev. They studied nonexistence results for entire weak, in the Sobolev sense, solutions of equations without dependence on the gradient. Mitidieri and Pohozaev used a priori estimates obtained with a careful choice of test functions. Important existence and nonexistences results for equations like $(P\pm)$ when the gradient term has a particular (power) form  are due to Filipucci, Pucci and Rigoli \cite{ filippucci2008non, filippucci2009entire, filippucci2009weak}. They used refined techniques based on comparison principles and the solutions they considered  were weak in the sense of entire $C^{1}$-distributional solutions. Besides, the work by Farina and Serrin \cite{farina2011entire} contains important contributions for equations containing the $p$-Laplacian, with nonlinearities which behave like products (as opposed to sums and differences) of $u$ and its gradient, possibly dependent on $x$ with conditions required only for large radii.

However, to our knowledge,  when the problems involve a term which depends on the gradient, in all references, it has power growth or the non linearity behaves exactly as a product of a term in $u$ and a term which depends on the gradient.
Motivated  by these results we attempt to get a generalization of the results in \cite{BB} to the case of the $p$-Laplacian. We proceed by obtaining some generalized Keller-Osserman (KO) integral conditions and deal  with comparison principles too.

 We recall that the Keller-Osserman integral conditions where developed independently by Keller in \cite{keller} and Ossermann in  \cite{ossi}, to obtain necessary and sufficient conditions for the existence of positive solutions of
 \bet
 \Delta u =f(u) \,\,\,\, \mbox{in}\,\, \rn,
 \eet
 with $f$ as in \EQ{fgcond}.

We need some different techniques from those used for the Laplacian, because of the intrinsic properties of the $p-$Laplacian operator. Since it is strictly elliptic in $\rn\setminus \left\{0\right\}$ only when $1<p\leq 2$ and degenerate when $p > 2$ we had to deal with this degeneracy. In addition, for the case of $(P_{+})$, we manage to prove the nonexistence of a positive solution in the weak Sobolev  sense, given in the following definition.
\begin{definition}[Solution of $(P+)$]
\label{def:defsol}
 We say that $u\in \ds{W_{\loc}^{1,p} (\rn)}$ is a weak subsolution of $(P\pm)$ in $\rn$ if
\begin{equation}
\int_{\rn} \left|\nabla u\right|^{p-2}\nabla u\nabla \phi \diff x \leq \int_{\rn}-(f(u)\pm g(\left|\nabla u\right|))\phi \diff x
\end{equation}
for each $\phi \in C_{c}^{\infty}(\rn)$, $\phi\ge0$.
Analogously we can define a weak supersolution by presuming the opposite inequality. Finally we say that $u$ is a weak solution, in the Sobolev sense,  if it is subsolution and  supersolution at the same time.
\end{definition}
\begin{rmk}\label{rem:rem1}
Specifically, our nonexistence results are valid for weak solutions which belong to the class $\ds{W^{1,\infty}_{\loc}(\rn)}$. However, as we explain later in Remark \ref{rem:weaklp}, we can still obtain a nonexistence result for subsolutions in $\ds{W_{\loc}^{1,p}(\rn)}$, if a comparison principle were proven for solutions in $\ds{W_{\loc}^{1,p}(\rn)}$ of  $(P+)$.
\end{rmk}

The following theorem contains our existence and non-existence statements for the problem $(P+)$. The non-existence result says that if a condition on $f$, resp. $g$, which guarantees lack of solutions for the problem with $g=0$, resp. $f=0$, is valid, then there are no solutions for the full problem.
In the existence result we establish how the value of $p$ influences the solvability of the problem. For $p$ larger or smaller than $2$ we obtain an explicit relation between the functions $f$ and $g$, which permits to us  to conclude about the existence.

\begin{thm}[Existence theorem for $(P+)$]\label{thm:first}
Let $f$,\;$g$ be functions satisfying \EQ{fgcond} and $\displaystyle{F(s)=\int_{0}^{s}f(t)\diff t}$.
\begin{itemize}
\item[(i)] If either
\be\label{eq:condintext}
\int^\infty_1{\frac{1}{(F(s))^{\frac{1}{p}}}}\diff s<+\infty\;\;\;\;\; \mbox{or}\;\;\;\;
\int_{1}^{\infty}\frac{s^{p-2}\diff s}{g(s)}<+\infty
\ee
 then any non negative weak subsolution in the space $\ds{W^{1,\infty}_{\loc}}(\rn)$ of
 \be\label{eq:Pmas}
\plap u = f(u)+g(\left|\nabla u\right|)\tag{$P_{+}$}\;\; \mbox{in} \;\; \rn
\ee
is identically zero.
\item [(ii)] If

\noindent\begin{minipage}{0.4\textwidth}
\be
\int^\infty_1{\frac{1}{(F(s))^{\frac{1}{p}}}}\diff s=+\infty \label{eq:kof}
\ee
    \end{minipage}%
    \begin{minipage}{0.1\textwidth}\centering
    and
    \end{minipage}%
\noindent\begin{minipage}{0.4\textwidth}
\be
\int_{1}^{\infty}\frac{s^{p-2}\diff s}{g(s)}=+\infty \label{eq:kog}
\ee
    \end{minipage}\vspace{0.5cm}

and there exist numbers $A_0,\epsilon_0>0$ such that for all $A\geq A_0$, either
\be\label{eq:ass1}
\mbox{if\;} \;\; p\leq 2,\;\;\;\;\liminf_{s\rightarrow \infty} \frac{g(A F(s)^{\frac{1}{p}})}{A^{p}f(s)}>\frac{1}{p}+\epsilon_{0}
\ee
or
\be\label{eq:ass2}
\mbox{if\;} \;\; p\geq 2,\;\;\;\;\limsup_{s\rightarrow \infty}\displaystyle   \frac{g(A F(s)^{\frac{1}{p}})}{A^{p}f(s)}<\displaystyle\frac{1}{p}-\epsilon_{0}
\ee
then $(P_{+})$  has at least one positive solution.
\end{itemize}
\end{thm}
The integral conditions  represent generalized (KO) integral conditions. Observe that in the problem $(P+)$ growth of $f$ and $g$ leads to   growth in the $p$-Laplacian, hence, at least morally, to growth of $u$. In this way,
\EQ{kof} and \EQ{kog} describe how  quickly $f(s)$ and $g(s)$ can grow as $s$ tend to infinity, so that the positive solution does not blow up at some finite point.

To state a few examples,  the generalized (KO) integral condition associated to $f$, \EQ{kof}, is satisfied by standard functions whose growth at infinity does not exceed that of $t^{q} \;,q\leq p-1$ or $t^{p-1} (\log t)^q\;,0<q\leq p$. The (KO) integral condition \EQ{kog} associated to $g$ is valid, for example, by functions whose growth at infinity does not exceed that of $t^{q} \;,q\leq p-1$ or $t^{p-1} (\log t)^q\;,0<q\leq 1$.

The hypotheses  \EQ{ass1} and \EQ{ass2} set a comparison between $g\circ F^{\frac{1}{p}}(s)$ and $f(s)$ for large values of $s$ which depends on the values of $p$. We note that for most standard functions the limits in \EQ{ass1} and \EQ{ass2} are zero or infinity, so these hypotheses are easy to verify.
We also note that if $g\circ F^{\frac{1}{p}}$ grows no faster than $f$ at infinity we cannot conclude anything about the existence of positive solutions if $p<2$ even if \EQ{kof} and \EQ{kog} are satisfied. Analogously if $g\circ F^{\frac{1}{p}}$ grows strictly faster than $f$, and $p>2$.
\begin{rmk}\label{rem:altcond}
Assumptions \EQ{ass1} and \EQ{ass2} can be avoided in the particular case where $g(s)$ has at most $s^{p-1}$ growth as $s$ tends to infinity. This will be explained later.
\end{rmk}
The treatment for $(P-)$ is somewhat different and in some way simpler, because we don't need the additional conditions \EQ{ass1} and \EQ{ass2} to conclude about the existence. In addition, \EQ{kog} is unnecessary to obtain existence.

 We resume our results on $(P-)$ in the next theorem.
\begin{thm}[Existence theorem for $(P-)$]\label{thm:second}
Let $f$ and $g$ be functions that satisfy \EQ{fgcond}.
\begin{itemize}
\item [(i)] If
\begin{equation}\label{eq:pmecond1}
\int^{\infty}_{1}\frac{1}{\Gamma^{-1}(F(s))}\diff s < \infty
\end{equation}
with $\Gamma$ defined by
\be\label{eq:eqgama}
\Gamma(s)=\int^{2s}_{0} g(t) \diff t + \frac{p-1}{p} c s^{p}
\ee
then any subsolution of $(P_{-})$ vanishes identically.
\item [(ii)] If
\be\label{eq:pminsecond}
\int^{\infty}_{1}\frac{1}{(F(s))^{\frac{1}{p}}}\diff s = \infty \;\;\; \mbox{or} \;\;\; \int^{\infty}_{1}\frac{1}{g^{-1}(f(s))}\diff s = \infty
\ee
then $(P_{-})$ admits at least one positive solution.
\end{itemize}
\end{thm}
Observe that, by the definition of $\Gg$, if \EQ{pmecond1} is satisfied then none of the assumptions in \EQ{pminsecond} is possible. Analogously if at least one of the integral condition in \EQ{pminsecond} is valid, then \EQ{pmecond1} is not possible. On the other hand nothing can be said about \EQ{pmecond1} if \EQ{pminsecond} is not satisfied.

In the end, it is worth noting that it is only technical to extend  all the above results to non-autonomous equations, when positive continuous weight functions (depending on $|x|$) multiply $f$ and $g$. This naturally leads to  modifications in the generalized (KO) integral conditions. We also observe that, by applying for instance Young's inequality, the case of a right-hand side which behaves as a product of functions of $u$ and its gradient can be reduced to the problem $(P+)$.
\medskip

\noindent{\it Acknowledgement.} I thank my Ph.D. adviser Prof. B. Sirakov for suggesting the problem and many helpful remarks, which improved the presentation of the paper.

\section{Preliminaries}
\subsection{Some features of the $p-$Laplacian}
The $p-$Laplacian operator reduces to the  Laplacian when $p=2$.
However, the $p-$Laplacian has different structural properties from  the Laplacian as $p\not=2$, as we quickly recall.

Set
$$\funcionesc{A}{\rn}{\rn}{\xi}{ A(\xi):=\left|\xi\right|^{p-2}\xi}$$
To analyze the ellipticity of \EQ{eqlap} we compute the Jacobian matrix associated to $A$:
\begin{equation*}
\partial_{\xi}A(\xi)=\left|\xi\right|^{p-2}\left(I_{n}+\frac{p-2}{\left|\xi\right|^{2}}\xi\otimes\xi\right).
\end{equation*}
The  eigenvalues of this matrix are $\left|\xi\right|^{p-2}$ and $(p-1)\left|\xi\right|^{p-2}$ with $\xi\neq 0$ if $p<2$.
When $1<p<2$ the $p-$Laplacian  is a {\it singular} operator, while for $p>2$ it is a {\it degenerate} operator. Observe that the zeros of the gradient are the singularities of $\Delta_p u$. We can say nothing about the ellipticity at these points.

The singular/degenerate character of the $p-$Laplacian is  the main cause of the difficulties in establishing comparison principles. We will use the following result from \cite{pp}.
\begin{thm}[Comparison principle]
\label{compprinc2}
Assume that $B=B(x,z,\xi)$ is locally Lipschitz continuous with respect to $\xi$ in $\Omega\times\R\times\rn$ and is non-increasing in the variable $z$.  Let $u$ and $v$ be solutions of class $\ds{W^{1,\infty}_{\loc}(\Omega)}$ of
\begin{equation*}
\plap u+ B(x,u,Du)\geq 0 \;\mbox{in} \;\Omega\;\; ,\;\;\; \plap v+ B(x,v,Dv)\leq 0 \;\mbox{in}\; \Omega,
\end{equation*}
where $p>1$. Suppose that
\bet
\ess\inf_{\Omega}\left\{\left|Du\right|+\left|Dv\right|\right\}>0.
\eet
 If $u\leq v+M$ in $\partial\Omega$ where $M\geq 0$ is constant then $u\leq v+M$ in $\Omega$.
\end{thm}
A complete proof can be find in \cite{pp}, corollary 3.6.3.
\begin{rmk}
The condition $\ess\inf_{\Omega}\left\{\left|Du\right|+\left|Dv\right|\right\}>0$ is vital for the comparison principle to be satisfied. In fact, if a comparison principle is satisfied we obtain uniqueness results. On the other hand, as observed in \cite{pp}, the problem
\bet
\begin{cases}
\Delta_{4}u+\abs{Du}^{2}&=0\;\;\;\mbox{in}\;\; B_{R}\subset \mathbb{R}^{2} \\
\;\;\;\;\;\;\;\;\;\;\;\;\;\;\;\;\;\;u&=0 \;\;\;\mbox{on}\;\; \partial B_{R}
\end{cases}
\eet
admits two solutions $u(x)=0 $ and $v(x)=\frac{1}{8}(R^{2}-|x|^{2})$ in $B_{R}$. It is obvious that $\left|Du\right|+\left|Dv\right|=0$ at zero, and consequently $\ess\inf_{\Omega}\left\{\left|Du\right|+\left|Dv\right|\right\}=0$ for these solutions.

We will be able to apply the comparison principle thanks to  the hypotheses for $f$ and~$g$, which we made in \EQ{fgcond}.
\end{rmk}

\subsection{Associated ODE problem}
Since we are going to use generalized (KO) conditions we need to deal with the radial version of our problem. Thus, in this section we focus our attention on the associated ODE.

Take $u(x)=v(\left|x\right|)=v(r)$ in \EQ{Pmasmenos} and add the initial conditions $u(0)=v(0)=v_{0}>0$ and $v'(0)=0$, which are consistent with our goals. Thus, the problem
\pvi{eq:pmmaedo}{\left(r^{n-1}\vlp\right)'}{r^{n-1}\left(f(v)\pm g(v')\right)}
corresponds to the radial version of \EQ{Pmasmenos} with fixed initial data.

More precisely, if we compute the $p-$Laplacian of a radial function $u(x)=v(r)$ we obtain the one-dimensional operator
\begin{equation}
\label{eq:uso1}
\sign (v'(r))\left(r^{n-1}\left|v'(r)\right|^{p-1}\right)',
\end{equation}
and of course  $g(|\nabla u|)=g(\left|v'\right|)$.

In \EQ{pmmaedo}  the equation is written in a simplified form because, as we will see next,  $v$ is positive and $v'$ is such that $\vl(0)=0$ and $\vl(r)>0$ for all $r>0$.

It is essentially known that for a ODE problem like \EQ{pmmaedo}  the non-negativity of the solutions and its derivatives can be  deduced a priori. We give a full proof of this fact, for the reader's convenience.
\begin{lem}[A priori properties of the radial solutions]
\label{lem:prad}
Let $v=v(r)$ be a solution of
\begin{equation}\label{eq:edoorg}
\begin{cases}
\sign(v')\left(r^{n-1}\left|v'\right|^{p-1}\right)'&=r^{n-1}\left(f(v)\pm g(\left| v'\right|)\right) \\
\;\;\;\;\;\;\;\;\;\;\;\;\; \;\;\;\;\;\; \;\;\;\;\;\; \;\;\;\; v(0)&=v_{0}>0\\
\;\;\;\;\;\;\;\;\;\;\;\;\; \;\;\;\;\;\; \;\;\;\;\;\; \;\;\; v'(0)&=0
\end{cases}
\end{equation}
in some interval $\left[0, R\right]$ with $0<R<\infty$. Then $v> 0$, $v'> 0$, $v''\geq 0$ and $v(r)\leq v_{0}+ R v'(r)$, for all $r\in(0,R)$.
\end{lem}
\begin{proof}[Proof of Lemma \ref{lem:prad}]
By an simple computation it is easy to see that the problem \EQ{edoorg} can be rewritten as
\begin{equation}\label{eq:edoorgmod}
\begin{cases}
\left(\left|v'\right|^{p-2}v'\right)'+\frac{n-1}{r}\left|v'\right|^{p-2}v'&=f(v)\pm g(\left|v'\right|) \\
\;\;\;\;\;\;\;\;\;\;\;\;\; \;\;\;\;\;\;\;\;\;\;\;\;\;\;\;\;\;\;\;\;\;\;\;\; v(0)&=v_{0}>0\\
\;\;\;\;\;\;\;\;\;\;\;\;\; \;\;\;\;\;\;\;\;\;\;\;\;\;\;\;\;\;\;\;\;\;\;\ v'(0)&=0.
\end{cases}
\end{equation}
First we deal with the signs of $v'$ and $v$.
By letting $r\rightarrow 0$ we obtain that
\begin{equation*}
\lim_{r\rightarrow 0}\left(\left|v'\right|^{p-2}v'\right)'(r)=\frac{1}{n} f(v_{0})>0,
\end{equation*}
which implies that $\left(\left|v'\right|^{p-2}v'\right)'(r)>0$ for all $r>0$ close enough to zero. Then obviously $\left|v'\right|^{p-2}v'(r)>0$ for all $r>0$ close enough to zero as well. Consequently $v'(r)>0 $ for all $r>0$ close  to zero. Now we need to study if this behavior is the same for all $r>0$.

Suppose that there exists $r_{1}>0$ such that $\left|v'\right|^{p-2}v'(r)>0$ in $(0,r_{1})$ and $\left|v'\right|^{p-2}v'(r_{1})=0$. Then we would have
$\left(\left|v'\right|^{p-2}v'\right)'(r_{1})\leq 0$. On the other hand, by using the equation in \EQ{edoorgmod} we obtain that $\left(\left|v'\right|^{p-2}v'\right)'(r_{1})> 0$, a contradiction. Hence $v'(r)>0$ for all $r>0$. Since $v(0)=v_{0}>0$ we also conclude that $v(r)>0$ for all $r>0$.

Knowing the signs of $v$ and $v'$ we rewrite our equation as
\begin{equation*}
\left((v')^{p-1}\right)'+\frac{n-1}{r}(v')^{p-1}=f(v)\pm g(v').
\end{equation*}
In order to deal with the sign of $v''$ we separate the argument in two cases, according to the sign $+$ or $-$ in the non-linearity.
\begin{itemize}
    \item [Case 1.] Positive sign in the non-linearity.
\end{itemize}
\pvi{eq:edomp}{\vlpl+\frac{n-1}{r}\vlp}{f(v)+ g(v')}

Suppose that there exist  $\epsilon,\; r_{1}>0$ such that $((v')^{p-1})'(r)>0$ in $(r_{1}-\epsilon,r_{1})$ and $((v')^{p-1})'(r_{1})=0$. Taking $h>0$ sufficiently small  we have
\begin{equation*}
 ((v')^{p-1})'(r_{1}-h)+(n-1)\frac{(v')^{p-1}(r_{1}-h)}{r_{1}-h}\leq((v')^{p-1})'(r_{1})+(n-1)\frac{(v')^{p-1}(r_{1})}{r_{1}},
\end{equation*}
since $f(v(r))$ and $g(v'(r))$ are strictly increasing  for all $r>0$. Then,
\begin{eqnarray*}
(n-1)\left(\frac{(v')^{p-1}(r_{1})}{r_{1}}-\frac{(v')^{p-1}(r_{1}-h)}{r_{1}-h}\right)&\geq& ((v')^{p-1})'(r_{1}-h)-((v')^{p-1})'(r_{1})\\
&=&((v')^{p-1})'(r_{1}-h)>0.
\end{eqnarray*}
Now dividing by $h$ and letting it tend to zero
\begin{eqnarray*}
0&<& \lim_{h\rightarrow 0}(n-1)\frac{1}{h}\left[\frac{(v')^{p-1}(r_{1})}{r_{1}}-\frac{(v')^{p-1}(r_{1}-h)}{r_{1}-h}\right]\\
&=&(n-1)\left[\frac{(v')^{p-1}}{r}\right]'\bigg|_{r=r_{1}}\\
&=&(n-1)\left[\frac{((v')^{p-1})'(r_{1})r_{1}-(v')^{p-1}(r_{1})}{r_{1}^{2}}\right]\\
&=&-(n-1)\frac{(v')^{p-1}(r_{1})}{r_{1}^{2}}.
\end{eqnarray*}
 Thus  $(v')^{p-1}(r_{1})< 0$, and since $r_{1}$ is arbitrary we obtain a contradiction with the fact that $v'(r)>0$ for all $r>0$.

\begin{itemize}
    \item [Case 2.] Negative sign in the non-linearity.
\end{itemize}
\begin{equation}\label{eq:edomm}
\begin{cases}
\left((v')^{p-1}\right)'+\frac{n-1}{r}(v')^{p-1}&=f(v)- g(v') \;\;\; \mbox{in}\;\;(0,\iy) \\ \;\;\;\;\;\;\;\;\;\;\;\;\; \;\;\;\;\;\;\;\;\;\;\;\;\;\;\;\;\;\; v(0)&=v_{0}>0\\
\;\;\;\;\;\;\;\;\;\;\;\;\; \;\;\;\;\;\;\;\;\;\;\;\;\;\;\;\;\; v'(0)&=0.
\end{cases}
\end{equation}
Suppose that $((v')^{p-1})'(r_{1})<0$ for some $r_{1}>0$ . Let
\begin{equation*}
r_{2}=\inf \left\{\tilde{r}\;:\;((v')^{p-1})'(r)<0\; \mbox{in}\; (\tilde{r},r_{1})\right\}.
\end{equation*}
 Since $\lim_{r\rightarrow 0}((v')^{p-1})'(r)> 0$ we have that $r_{2}>0$ and $((v')^{p-1})'(r_{2})= 0$. Moreover $((v')^{p-1})'(r)< 0$ for $r>r_{2}$  sufficiently close to $r_{2}$. This implies that $(v')^{p-1}(r)$ is decreasing for $r>r_{2}$ sufficiently close to $r_{2}$. On the other hand, since $v$ is increasing
\begin{equation*}
((v')^{p-1})'(r)=f(v(r))-\frac{n-1}{r}(v')^{p-1}-g(v')
\end{equation*}
is increasing. Thus $((v')^{p-1})'(r_{2})=0$ implies that $((v')^{p-1})'(r)>0$ if $r>r_{2}$ close to $r_{2}$, a contradiction, and we are done.

Also, since  $v''(r)>0$ for all $r>0$ we have that $v'(r)$  is non decreasing for  $0<r<R$, then
\begin{equation*}
v(r)=v_{0}+\int_{0}^{r} v'(s) \diff s\leq v_{0}+R v'(r).
\end{equation*}
\end{proof}

\section{Local existence for the radial problem}

  We will now study the existence of solution of \eqref{eq:pmmaedo} in a right neighborhood of zero.  Because of the presence of the gradient in the right hand side we will need to apply topological tools, in particular,  the Leray-Schauder theorem. This dependence in the gradient makes it necessary to bound not only the function but also its derivative. Consequently, in this way we can only show the existence of solution  in a neighborhood of zero. Furthermore, even this is not immediate when $p\neq 2$ because of  the expression involving the exponent $p-1$ inside the derivative on the left hand side.

In the next lemma we prove the existence of solution in a neighborhood of zero.  Related results have appeared for instance in \cite{chen2001singular} and \cite{bueno2012positive}.

\begin{lem}[Existence of solution in a neighborhood of zero]
\label{lem:solmax}
Let $f$ and $g$ be continuous increasing functions. Then there exists $0< r_{1}\leq 1$ such that  the problem \EQ{pmmaedo} has a positive  solution $v(r)\in C^{2}[0, r_{1}]$.
\end{lem}
This lemma will be obtained   by observing that
finding a solution for the problem is equivalent to finding a fixed point of some well chosen integral operator. Then using the Arzel\'{a}-Ascoli theorem and the dominated convergence theorem we prove the continuity and compactness of this operator. Last we apply the Leray-Schauder fixed point theorem  (corollary $11.2$ in \cite{truddi}), to this operator in a well chosen closed, convex  and bounded set.

\begin{proof}
For ease of notation set
\begin{equation*}
    H(t):=t^{p-1}, \;\; t>0
\end{equation*}
and
\begin{equation*}
    F(v,v'):= f(v)\pm g(v').
\end{equation*}
With this notation we rewrite the radial version of our problem as
\pvi{eq:probexit}{\left(r^{n-1} H(\vl)\right)'}{F(v,v')}
We consider the Banach space $X=C^{1}([0,r_{1}])$ with the associated norm
\begin{equation*}
\left\|u\right\|_{C^{1}([0,r_{1}])}=\left\|u\right\|_{L^{\infty}([0,r_{1}])}+\left\|u'\right\|_{L^{\infty}([0,r_{1}])}
\end{equation*}
and $0<r_{1}<1$ that will be chosen below. Let the integral operator
\begin{equation*}
\begin{split}
T:\ C^{1}([0,r_{1}]) &\rightarrow C^{1}([0,r_{1}])
\end{split}
\end{equation*}
be defined as
\be
\label{eq:operator}
T v(r)= v_{0}+\int^{r}_{0}\exist{s}{s}{v(t)}{\vl(t)}\diff s.
\ee
Observe that solving \EQ{probexit} is equivalent to finding a fixed point of this integral operator. From now on we focus on showing the existence of such a fixed point.

 Let $\left\{v_{k}\right\}_{k\in \N}$ be a bounded sequence in $C^{1}([0,r_{1}])$,
\begin{equation*}
\left\|v_{k}\right\|_{C^{1}([0,r_{1}])}\leq M, \;\;\; \forall k\in \Z  .
\end{equation*}
By the monotonicity properties of $f$ and $g$, for all  $r\in [0,r_{1}]$
\begin{equation*}
f(v_{k}(r))\leq f(M):= k_1, \qquad
g(v'_{k}(r))\leq g(M):= k_{2},
\end{equation*}
and  $F(v,v')\leq k_{1}+k_{2}$.

Then, using that $H^{-1}$  is strictly increasing and the fact that $0\leq r <r_{1}<1$ we have
\begin{eqnarray*}
T v_{k}(r)&=& v_{0}+\int^{r}_{0}H^{-1}\left(\int^{s}_{0}\left(\frac{t}{s}\right)^{n-1} F(v_{k}(t),v_{k}'(t))\diff t\right)\diff s\\
&\leq & v_{0}+\int^{r_{1}}_{0}H^{-1}\left((k_{1}+k_{2})\int^{s}_{0}\left(\frac{t}{s}\right)^{n-1} \diff t\right)\diff s\\
&\leq & v_{0}+r_{1}H^{-1}\left((k_{1}+k_{2})\int^{r_{1}}_{0}\diff t\right)\\
&\leq & v_{0}+H^{-1}\left(k_{1}+k_{2}\right).
\end{eqnarray*}
Consequently $Tv_{k}(r)$ is uniformly bounded in the sup norm.

We also have
\begin{eqnarray}
\label{eq:tderv}
(Tv_{k})'(r)&=& H^{-1}\left(\int^{r}_{0}\left(\frac{t}{r}\right)^{n-1} F(v_{k}(t),v_{k}'(t))\diff t\right)\nonumber\\
&\leq & H^{-1}\left((k_{1}+k_{2})\int^{r}_{0}\left(\frac{t}{r}\right)^{n-1} \diff t\right)\nonumber\\
&\leq & H^{-1}\left((k_{1}+k_{2})\int^{r}_{0}\diff t\right)\nonumber\\
&\leq & H^{-1}\left((k_{1}+k_{2})\;r\right)\\
&\leq & H^{-1}\left(k_{1}+k_{2}\right),\nonumber
\end{eqnarray}
showing that $(Tv_{k})'(r)$ is uniformly bounded in the sup norm. Then $Tv_{k}(r)$ is equicontinuous.\\
Since we deal with the Banach space $\left(C^{1},\left\|\cdot\right\|_{C^{1}}\right)$ we need to prove the equicontinuity also for $(Tv_{k})'(r)$.
Thus, we obtain ($\tilde{F}:= F(\vk,\vkl)$)
\bet
(Tv_{k})''(r)=\frac{1}{p-1}\left( \int_{0}^{r}\left(\frac{t}{r} \right)^{n-1} \tilde{F} \diff t\right)^{\frac{2-p}{p-1}}
\left(\tilde{F}+\frac{1-n}{r}\int_{0}^{r}\left(\frac{t}{r}\right)^{n-1}\tilde{F} \diff t \right).
\eet
Letting $\phi_{k}=(Tv_{k})'$
\begin{equation}\label{eq:segdopert}
    \left|(Tv_{k})''(r)\right|=\frac{1}{p-1}\left[\left|\phi_{k}(r)\right|^{2-p} F(\vk(r),\vkl(r))+\frac{n-1}{r}\left|\phi_{k}(r)\right|\right].
\end{equation}
To study the  uniform boundedness of this derivative we need to pay attention to each of the two terms in this sum.  Observe that in the first term appears the exponent $2-p$, which  leads to separation in two cases, according to the degeneracy character of the operator. As for the second term, note that although $\fik$ is uniformly bounded we can not immediately conclude because of the term $1/r$.

Besides, from \EQ{tderv} we know that
\bet
\fik(r)=(Tv_{k})'(r)\leq H^{-1}(k_{1}+k_{2}) r^{\frac{1}{p-1}}
\eet
which implies that
\be
\label{eq:bdfi}
\frac{(n-1)\fik(r)}{r}\leq (n-1) H^{-1}(k_{1}+k_{2}) r^{\frac{2-p}{p-1}}.
\ee
Thus we also need to consider the values of $p$ for the second term of the sum.
\begin{itemize}
    \item[Case I.]  $1<p\leq 2$.
\end{itemize}
If $1<p\leq 2$, we have that $2-p\geq 0$.  Since $\phi_{k}$ is uniformly bounded as we saw earlier the first term on the right hand side of \EQ{segdopert} is also uniformly bounded. On the other hand by \EQ{bdfi} the second term is uniformly bounded too.  Then $(Tv_{k})'$ is Lipschitz continuous uniformly with respect to $k$, and consequently equicontinuous.
\begin{itemize}
    \item[Case II.] $p>2$.
\end{itemize}
If $p> 2$ we have that $2-p< 0$. Then the Lipschitz continuity is lost and we can not conclude about the equicontinuity of $(Tv_{k})'$ as above. We will prove that $(Tv_{k})'$ is $\alpha-$H\"{o}lder continuous with $\displaystyle{\alpha=\frac{1}{p-1}<1}$, to conclude about the equicontinuity

Let $\lambda_{k}(r)=\displaystyle{\int^{r}_{0}\left(\frac{t}{r}\right)^{n-1} F(v_{k}(t),v_{k}'(t))\diff t}$, then
\begin{equation*}
    (Tv_{k})'(r)=\left(H^{-1}\circ\lambda_{k}\right)(r),
\end{equation*}
where,
$H^{-1}(t) = t^{\frac{1}{p-1}}$, $t>0$, is $\alpha-$ H\"{o}lder continuous with $\displaystyle{\alpha=\frac{1}{p-1}}$. On the other hand we observe that $\lambda_{k}(r)\in C^{2}([0,r_{1}])$ with
\begin{equation*}
    \lim_{r\rightarrow 0^{+}} \lambda_{k}(r)= 0 \;\;\mbox{and}\;\; \lim_{r\rightarrow 0^{+}} \lambda_{k}'(r)=\frac{F(v_{0},0)}{n}=\frac{f(v_{0})}{n}
\end{equation*}
Thus, using the mean value theorem we obtain that exist $L>0$ such that
\be
\label{eq:lbdlips}
    \left|\lambda_{k}(r)-\lambda_{k}(\ell)\right|\leq L\left|r-\ell\right|.
\ee
This means that $\lambda_{k}$ is locally Lipschitz continuous uniformly in $k$. Then, using the H\"{o}lder continuity of $H^{-1}(t)$ and \EQ{lbdlips}
\begin{eqnarray*}
 \left|(Tv_{k})'(r)-(Tv_{k})'(\ell)\right|&=& \left|(H^{-1}\circ \lambda_{k})(r)-(H^{-1}\circ \lambda_{k})(\ell)\right|\\
&\leq & \left|\lambda_{k}(r)- \lambda_{k}(\ell)\right|^{\frac{1}{p-1}}\\
&\leq & L^{\frac{1}{p-1}}\left|r- \ell\right|^{\frac{1}{p-1}}.
\end{eqnarray*}
Consequently $(Tv_{k})'(r)$ is H\"{o}lder continuous with exponent $\frac{1}{p-1}$ uniformly in $k$, which implies that $Tv_{k}(r)$ is equicontinuous also for $p>2$.

To prove the continuity of $T$ we observe that, if $ v_{k} \rightarrow v $ uniformly in $C^{1}(\left[0,r_{1}\right])$, then by the dominated convergence theorem for any sub-sequence $\left\{ v_{k_{j}} \right\}$ of $\left\{ v_{k} \right\}$ there is a further subsequence (still denoted by $v_{k_{j}}$), such that  $ T v_{k_{j}} \rightarrow T v$ in $C^{1}(\left[0,r_{1}\right])$. Thus follows the continuity of the integral operator $T$.

Consequently by the the Arzel\'{a}-Ascoli theorem we conclude that $T$ is  compact and continuous operator.

At last we need to chose an appropriate closed, convex and bounded subset such that the operator $T$ maps this set into itself.
Given $0<m_{2}<v_{0}$, $m_{3}>0$, by the continuity of $v$ and $v'$ we can choose a sufficiently small constant $r_{0}>0$ such that
\begin{equation*}
    \max_{0\leq r\leq r_{0}}\left\{\left|v(r)-v_{0}\right|\right\}\leq m_{2}, \qquad\mbox{and}\qquad
  \max_{0\leq r\leq r_{0}}\left\{\left|v'(r)\right|\right\}\leq m_{3}.
\end{equation*}
Denote, in case of $(P_+)$
\[
r_{\am}:=\max\left\{ (\frac{p}{p-1})^{\frac{p-1}{p}}\frac{ n^{\frac{1}{p}},\mdo^{\frac{p-1}{p}}}{(F(v_{0}+\mdo, \mtre))^{\frac{1}{p}}}\;\;,\;\; \frac{n\;\mtre^{p-1}}{F(v_{0}+\mdo, \mtre)}\right\}
\]
and in case of $(P_-)$
\[
r_{\am}:=\max\left\{(\frac{p}{p-1})^{\frac{p-1}{p}}\frac{ n^{\frac{1}{p}}\;\mdo^{\frac{p-1}{p}}}{(f(v_{0}+\mdo))^{\frac{1}{p}}}\;\;,\;\;\frac{n\;\mtre^{p-1}}{f(v_{0}+\mdo)}\right\}.
\]
Also denote the closed convex and bounded subset of the space $C^{1}(\left[0,r_{\alpha}\right])$ by
\begin{equation}
\label{eq:subcjto}
    B_{\alpha}=\left\{\phi\in C^{1}(\left[0,r_{\alpha}\right]); \;\; \left|\phi(s)-v_{0}\right|\leq m_{2}\;
,\;\left|\phi'(s)\right|\leq m_{3}\quad
\mbox{for all} \; 0\leq s\leq r_{\alpha}\right\}.
\end{equation}
We need to show that $T$ maps $B_{\alpha}$ into $B_{\alpha}$. For this we are going to show that $$\left|T\phi(r)-v_{0}\right|\leq m_{2}, \qquad\mbox{and}\qquad\left|(T\phi)'(r)\right|\leq m_{3}.$$
Since $f$ and $g$ are increasing we have  in case of $(P+)$ that $F(\phi(t), \phi'(t))\leq F(v_{0}+\mdo,\mtre)$ and in case of $(P-)$ that $F(\phi(t), \phi'(t))\leq F(v_{0}+\mdo,0)=f(v_{0}+\mdo)$. Also, using the monotonicity of $H^{-1}$ we have in case of $(P+)$
\begeqnat
\left|T\phi(r)-v_{0}\right|&\leq&\int^{r}_{0}\left|\exist{s}{s}{\phi(t)}{\phi'(t)}\right|\diff s\\
&\leq&\int^{r}_{0}H^{-1}\left(F(v_{0}+\mdo,\mtre)\frac{s}{n}\right)\diff s\\
&=& \frac{p-1}{p}\left(\frac{F(v_{0}+\mdo,\mtre)}{n}\right)^{\frac{1}{p-1}}r^{\frac{p}{p-1}}\leq \mdo.
\eeqnat
Analogously for the derivative
\begeqnat
\left|(T\phi)'(r)\right|&\leq&\left|\exist{s}{s}{\phi(t)}{\phi'(t)}\right|\\
&\leq&F(v_{0}+\mdo,\mtre)^{\frac{1}{p-1}} H^{-1}\left(\int^{r}_{0}\left(\frac{t}{r}\right)^{n-1}\diff t\right)\\
&=& \left(\frac{F(v_{0}+\mdo,\mtre)}{n}\right)^{\frac{1}{p-1}} r^{\frac{1}{p-1}}\leq \mtre.
\eeqnat
In the same way for $(P-)$
\bet
\left|T\phi(r)-v_{0}\right|\leq \frac{p-1}{p}\left(\frac{f(v_{0}+\mdo)}{n}\right)^{\frac{1}{p-1}}r^{\frac{p}{p-1}}\leq \mdo
\eet
and for the derivative
\bet
\left|T'\phi(r)\right|\leq  \left(\frac{f(v_{0}+\mdo)}{n}\right)^{\frac{1}{p-1}} r^{\frac{1}{p-1}}\leq \mtre.
\eet
Thus $T(B_{\alpha})\subset B_{\alpha}$, this means that $T$ is a continuous, compact mapping from $B_{\alpha}$ into $B_{\alpha}$. By the Leray-Schauder theorem the integral operator $T$ has a fixed point $v(r)=Tv(r)$ in $B_{\alpha}$. Therefore the problem has a radial solution $v\in C^{1}([0,r_{\alpha}])$.\\
On the other hand, by the definition of the operator we have that
\begin{equation*}
   Tw(r)=u(r)\in C^{2}((0,r_{\alpha}))\cap C([0,r_{\alpha}])
\end{equation*}
for each $w\in C^{1}([0,r_{\alpha}])$,
in particular, for the fixed point $v$. Consequently the radial solution of the problem is such that $v(r)\in C^{2}((0,r_{\alpha}))\cap C([0,r_{\alpha}])$. Taking $r_{1}<r_{\alpha}$ we obtain the desired result.
\end{proof}

Note that the above nonlinear technique gives existence for the ODE only for some finite interval of $r$, or, equivalently, existence for $(P\pm)$ in some (possibly small) ball. As we are interested in the existence or not of non-negative solutions in the whole $\rn$, a first candidate for the existence is exactly this radial solution, provided we are able to check that it actually exists globally.

In the following lemma  we show explicitly, in case of $(P+)$, that under some hypotheses any radial solution of the problem exist on a maximal {\it finite} interval. So, under these hypotheses, the candidate fails. In the next section we are going to see that the opposite hypotheses are necessary for the existence.
\begin{lem}
\label{lem:lemageneral}
Let $f$ and $g$ satisfy \EQ{fgcond}. Then, the following statements hold.
\begin{itemize}
\item[(i)] Under the assumptions in \EQ{condintext} (\THM{first} $(i)$) any solution of
\pvi{eq:edoma}{\left(r^{n-1}(\vl)^{p-1}\right)'}{r^{n-1}\left(f(v)+ g(\vl)\right)}
 exists on a maximal interval $(0,R)$ where $0<R<\infty$ and it is such that
 \bet
 v'(r)\rightarrow \infty\;\; \mbox{as}\;\; r\rightarrow R.
 \eet
 \item[(ii)] Furthermore,
 \be
 \label{eq:vnoexplode}
 v(r)\xrightarrow[r\rightarrow R]{} \infty\;\Longleftrightarrow\;
\int^{\infty} _{1}\frac{s^{p-1}}{g(s)}\diff s=\infty.
 \ee
\end{itemize}
\end{lem}
\begin{rmk}
Note that the integral condition in \EQ{vnoexplode} and \EQ{kog} are not exhaustive. In fact for $g(t)=t^{q}$ where $q>p$ neither of them is satisfied.
\end{rmk}
\begin{proof}[Proof of item (i)]
Integrating from $0$ to $r$ the equation in \EQ{edoma}, by the monotonicity of $f$ and $g$ and the a priori properties in \LEM{prad} we obtain that
\begin{eqnarray*}
(v')^{p-1}&=&\frac{1}{r^{n-1}}\int_{0}^{r}{s^{n-1}\left(f(v(s))+g(v'(s))\right)}\diff s\\
          &\leq& \frac{r}{n}\left(f(v(r))+g(v'(r))\right).
\end{eqnarray*}
Then using the rewritten version  \EQ{edomp} of \EQ{edoma} results in
\begin{eqnarray*}
\left(\left(v'\right)^{p-1}\right)'&\geq&f(v)+g(v')-\frac{n-1}{n}\left(f(v(r))+g(v'(r))\right)\\
&=&\frac{1}{n}\left(f(v(r))+g(v'(r))\right).
\end{eqnarray*}
thus $\left(\left(v'\right)^{p-1}\right)'\geq \frac{1}{n}f(v(r))$.
Multiplying the last inequality by $\frac{p}{p-1}v'\geq 0$ we have
\begin{equation*}
\frac{p}{p-1}\left(\left(v'\right)^{p-1}\right)'v'\geq \frac{p}{p-1}\frac{1}{n}\left(F(v)\right)',
\end{equation*}
then integrating from $0$ to $r$
\begin{eqnarray*}
\int^{r}_{0} p(v')^{p-1}v''\diff r&\geq&\frac{p}{n(p-1)}\int^{r}_{0} f(v(r))v'(r)\diff r\\
\int^{r}_{0} ((v')^{p})'\diff r&\geq&\frac{p}{n(p-1)}(F(v(r))-F(v_{0})).
\end{eqnarray*}
Thus
\bet
\vl\geq \left(\frac{p}{n(p-1)}\right)^{\frac{1}{p}}\left(F(v)-F(v_{0})\right)^{\frac{1}{p}}.
\eet
Dividing by $\left(F(v)-F(v_{0})\right)^{\frac{1}{p}}$ and taking the integral from $0$ to $r$ we obtain
\begin{equation*}
\int_{0}^{r}{\frac{v'(s)}{\left(F(v(r))-F(v_{0})\right)^{\frac{1}{p}}}\diff s}\geq \left(\frac{p}{n(p-1)}\right)^{\frac{1}{p}}r,
\end{equation*}
i.e,
\begin{equation}
\label{eq:ineqf}
\int_{v_{0}}^{v(r)}{\frac{1}{\left(F(s)-F(v_{0})\right)^{\frac{1}{p}}}\diff s}\geq \left(\frac{p}{n(p-1)}\right)^{\frac{1}{p}}r .
\end{equation}
Analogously
\begin{equation}
\label{eq:resintg}
\left(\left(v'\right)^{p-1}\right)'\geq \frac{1}{n} g(v'(r)).
\end{equation}
Dividing by $g(v')$ and integrating from some $r_{0}>0$ to $r>r_{0}$ we obtain
\begin{eqnarray*}
\frac{\left(\left(v'\right)^{p-1}\right)'}{g(v'(r))}&\geq & \frac{1}{n}\\
\int^{r}_{r_{0}}\frac{\left(\left(v'\right)^{p-1}\right)'}{g(v'(r))}\diff r&\geq&\frac{1}{n}(r-r_{0}),
\end{eqnarray*}
then
\begin{equation}
\label{eq:ineqg}
\int^{v'(r)}_{v'(r_{0})}\frac{s^{p-2}}{g(s)}\diff s\geq\frac{1}{n}(r-r_{0}).
\end{equation}
Since $F(v_{0})$ is a constant it does not affect the convergence of the integral in the inequality \EQ{ineqf}. Using the assumptions \eqref{eq:condintext} of the lemma we see that at least one of the integrals on the left side of \EQ{ineqf} and \EQ{ineqg} stays bounded independently of $r$. Thus $r$ is bounded above, consequently $R$ is finite.\\
Observe that the fact that the maximal interval be finite can be due to $v(r)\rightarrow\infty$ or $v'(r)\rightarrow\infty$ when $r$ tends to $R$. Since $v\;,\; v'$ are increasing and $v(r)\leq v_{0}+ R v'(r)$ (as we showed in \LEM{prad}) we conclude that $v'(r)\rightarrow\infty$ as $r$ tends to $R$.
\end{proof}
\begin{proof}[Proof of item (ii)]
Suppose that the integral in the right hand side is finite. Then multiplying \EQ{resintg} by $v'$ and dividing by $g(\vl)$ we obtain
\bet
(p-1)\vlp \frac{\vll}{g(\vl)}\geq \frac{1}{n} \vl
\eet
and integrating from any $r_{0}$ to $r$ such that $0<r_{0}<r<R$
\bet
    v(r)-v(r_{0})\leq n(p-1)\int^{\vl(r)}_{\vl(r_{0})}\frac{s^{p-1}}{g(s)}\diff s <\infty
\eet
as $r\to R$. This means that $v(r)$ is bounded.\\
On the other hand if we suppose that exist a finite constant $C>0$ such that $v(r)\leq C$ when $r\rightarrow R$, then in $(0,R)$
\bet
    \vlpl\leq \vlpl+\frac{n-1}{r}\vlp\leq f(C)+g(\vl).
\eet
Multiplying this inequality by $v'$ and dividing by $g(\vl)$ we obtain after integration from any $r_{0}$ to $r$ such that $0<r_{0}<r<R$
\begin{equation*}
\int^{v'(r)}_{v'(r_{0})}\frac{s^{p-1}}{g(s)}\diff s\leq c(v(r)-v(r_{0})).
\end{equation*}
Letting $r\rightarrow R$, since we know that $v'(r)\rightarrow \infty$ when $r\rightarrow R$ we obtain that the right hand side of \EQ{vnoexplode} is finite, as we want.
\end{proof}

\section{Existence and nonexistence of solutions for $P_{+}$}\label{sec:entsol}
In this section we prove  \THM{first} which ensures the existence or not of positive solutions to the equation
\begin{equation}
\plap u = f(u)+g(\left|\nabla u\right|).
\end{equation}
The first item is relative to the non existence of positive weak solutions. Comparing the solution of \EQ{pmmaedo} with eventual solutions of  our equation will lead us to the nonexistence result.
\begin{proof}[Proof of \THM{first} (i)]
Suppose by contradiction that the conclusion is false. That is, we assert that exist $w\in W_{\loc}^{1,\infty}(\rn)$ such that $w\geq 0$, $w \neq 0$, satisfies
\bet
\plap w \geq f(w)+ g(\left|\nabla w\right|).
\eet
 We can suppose without loss of generality that $w(0)>0$. In fact, since $w \neq 0$, there exists  a $x_{0}$ such that $w(x_{0})>0$. Then we can take $\widetilde{w}(x)=w(x+x_{0})$ instead of $w$.

Let $v(r)=v(\mo{x})=u(x)$ be a radial solution of \EQ{Pmas} defined in a maximal interval $(0,R)$ with $0<R<+\infty$ which we obtained in  \LEM{solmax}  and \LEM{lemageneral}, setting  $v_{0}=\frac{w(0)}{2}$ in \LEM{lemageneral}.

First we affirm that $v$ is bounded when $r\rightarrow R$. Indeed, if $v(r)\rightarrow \infty$ as $r\rightarrow R$ we can take $\epsilon>0$ sufficiently small so that $u > w$ in $\partial B_{R-\epsilon}(0)$. Then by using the comparison principle we have  that $u\geq w$ in $B_{R-\epsilon}(0)$ which contradicts  $v_{0}<w(0)$.

Now by applying the comparison principle once again in the whole ball $B_{R}(0)$ we obtain the existence of $\bar{x}\in\partial B_{R}(0)$ such that $w(\bar{x})>u(\bar{x})$. Then we can take $a>0$ such that the function $u_{a}=u + a$ satisfies $w(x)\leq u_{a}(x)$ for all $x\in \partial B_{R}(0)$ and $w(x_0)=u_a(x_0)$ for some $x_0\in \partial B_{R}(0)$. Specifically, $a=\sup_{x\in\partial B_{R}(0)}(w(x)-u(x))$.\\
We see that
\begin{equation*}
\plap u_{a}\leq f(u_{a})+ g(\nabla u_{a}) \ \ \ \ \mbox{in} \; B_{R}(0)
\end{equation*}
by the monotonicity of $f$. Then by the comparison principle we can infer
\begin{equation*}
w(x)\leq u_{a}(x) \; \forall \; x\in B_{R}(0).
\end{equation*}
Since $w(x_0)=u_a(x_0)$ this obviously implies
\begin{equation*}
\frac{ u_a(x_0+t\nu)-u_a(x_0)}{t}\ge \frac{ w(x_0+t\nu)-w(x_0)}{t},
\end{equation*}
for all small $t>0$, where $\nu=-x_0/|x_0|$ is the interior normal to the boundary of the ball at $x_0$. When $t\to0$, the left-hand side of this inequality tends to $-\infty$, by  \LEM{lemageneral} (i). Hence the right-hand side is unbounded, but this contradicts  $w\in W_{\loc}^{1,\infty}(\rn)$. Thus we can conclude that  the only nonnegative subsolution is the trivial one.
\end{proof}
\begin{rmk}
We recall that we adopt the usual definition of solution for the $p-$Laplacian, in the weak-Sobolev sense. If we suppose we have a weak solution in the sense of viscosity, which is less common, the statement is also satisfied.
\end{rmk}

\begin{rmk}\label{rem:weaklp}
Let us now give the detailed explanation of the statement given in \REM{rem1}.

In the above proof of the \THM{first} (i) we make fundamental use of the comparison principle. However, to our knowledge, for our type of equations  with dependence on the gradient the comparison principle is available only for functions in  $W_{\loc}^{1,\infty}(\rn)$. This forces us to make to assume our solutions are in this class. On the other hand, if a comparison principle in $W_{\loc}^{1,p}(\rn)$ were available, then we can affirm the same non-existence result as above for subsolutions in the Sobolev space $W_{\loc}^{1,p}(\rn)$, by imposing the additional integral condition
\be\label{eq:adcond}
\int_{1}^{\infty}{\frac{s^{2(p-1)}}{g(s)}}\diff s= \infty.
\ee
In fact, in this case we can repeat the above proof until the comparison between $w(x)$ and $u_{a}(x)$ in $B_{R}(0)$.

We will now show that under \eqref{eq:adcond} the radial function $u$ in the above proof is not in $W^{1,p}$ in a neighbourhood of the boundary of the ball, so again by comparison, the same must be true for $w$, which is a contradiction.

 Observe that if we multiply $(P+)$, i.e. the equation in \EQ{edomp}, by $\displaystyle{\frac{(v')^{p}}{g(v')}}$ we have
\begin{eqnarray*}
(v')^{p}&=&\frac{\left(\left(v'\right)^{p-1}\right)'\left(v'\right)^{p}}{g(v')}+\frac{n-1}{r}(v')^{p-1}\frac{(v')^{p}}{g(v')}-\frac{f(v)}{g(v')}(v')^{p}\\
        &=&(p-1)\frac{\left(v'\right)^{2(p-1)}}{g(v')}v''+\frac{n-1}{r}\frac{\left(v'\right)^{2p-1}}{g(v')}-\frac{f(v)}{g(v')}(v')^{p}.
\end{eqnarray*}
Now integrating from some $r_{0}>0$ to $r$, with $0<r_{0}<r<R$ results in
\be
\label{eq:pertlp}
\int_{r_{0}}^{r}{(v')^{p}}\diff r= (p-1)\int_{r_{0}}^{r}{\frac{\left(v'\right)^{2(p-1)}}{g(v')}v''}\diff r+(n-1)\int_{r_{0}}^{r}{\frac{\left(v'\right)^{2p-1}}{ r g(v')}}\diff r-\int_{r_{0}}^{r}{\frac{f(v)}{g(v')}(v')^{p}}\diff r .
\ee
Since  $v'(r)\rightarrow \infty$ as $r\rightarrow R$ taking $r\rightarrow R$ in \EQ{pertlp} and substituting  $s=v'$  in the first term in right hand side gives
\begin{equation}
\label{eq:pertlp1}
\int_{r_{0}}^{R}{(v')^{p}}\diff r= (p-1)\int_{v'(r_{0})}^{\infty}{\frac{\left(s\right)^{2(p-1)}}{g(s)}}\diff s+(n-1)\int_{r_{0}}^{R}{\frac{\left(v'\right)^{2p-1}}{ r g(v')}}\diff r-\int_{r_{0}}^{R}{\frac{f(v)}{g(v')}(v')^{p}}\diff r .
\end{equation}
Observe that by the monotonicity of $f$ and $g$ and since $v$ is bounded when $r\rightarrow R$ in $(r_{0},R)$
\begin{eqnarray*}
f(v(r))&\leq & f(v(R))\leq f(C)=C_{0}\\
g(v'(r))&\geq& g(v'(r_{0}))>0,
\end{eqnarray*}
then
\begin{equation*}
\int_{r_{0}}^{R}{\frac{f(v)}{g(v')}(v')^{p}}\diff r \leq \frac{C_{0}}{g(v'(r_{0}))}\int_{r_{0}}^{R}{(v')^{p}}\diff r.
\end{equation*}
Substituting this in the equation \EQ{pertlp1} we obtain that
\begin{multline}
\int_{r_{0}}^{R}{(v')^{p}}\diff r\geq (p-1)\int_{v'(r_{0})}^{\infty}{\frac{s^{2(p-1)}}{g(s)}}\diff s+(n-1)\int_{r_{0}}^{R}{\frac{\left(v'\right)^{2p-1}}{ r g(v')}}\diff r\\-\frac{C_{0}}{g(v'(r_{0}))}\int_{r_{0}}^{R}{(v')^{p}}\diff r .
\end{multline}
Then,
\begeqnat
\int_{r_{0}}^{R}{(v')^{p}}\diff r&\geq& \frac{1}{1+\frac{C_{0}}{g(v'(r_{0}))}}\left((p-1)\int_{v'(r_{0})}^{\infty}{\frac{s^{2(p-1)}}{g(s)}}\diff s+(n-1)\int_{r_{0}}^{R}{\frac{\left(v'\right)^{2p-1}}{ r g(v')}}\diff r\right)\\
&\geq & \frac{(p-1)}{1+\frac{C_{0}}{g(v'(r_{0}))}}\int_{v'(r_{0})}^{\infty}{\frac{s^{2(p-1)}}{g(s)}}\diff s.
\eeqnat
By the additional integral condition \EQ{adcond} we obtain that the integral on the right hand side in the last inequality is divergent, so
\begin{equation}
\label{eq:nolp}
\int_{r_{0}}^{R}{(v')^{p}}\diff r=\infty.
\end{equation}
\end{rmk}

\begin{rmk}
Observe that the statement \EQ{vnoexplode} in \LEM{lemageneral} is not used in  the above nonexistence proof. We included it in \LEM{lemageneral} in order to show that the radial solution does not always explode when $r$ tends to $R$. This is because of the dependence on the right hand side of the gradient. However  if $g$ is such that $\ds{\int_{1}^{\iy}\frac{s^{p-1}}{g(s)}\diff s=\iy}$, so $v$ actually explodes on the boundary of the ball, we obtain a much simpler proof of the nonexistence, since already the first comparison argument in the proof of \THM{first} (i) gives a contradiction.\medskip

We summarize the sufficient integral conditions that yield qualitative properties on $v$ and $v'$ for radial equations with dependence on the gradient. \\ \\
\begin{tabular}{l c r}
Result in Text & Integral Condition  & Properties from $v$ and $v'$ \\ \\
Not (KO) \EQ{kog} & $\displaystyle{\int^{+\infty}{\frac{s^{p-2}}{g(s)}}\diff s}< +\infty$ & $v'(r)\xrightarrow[r\rightarrow R]{}\infty$\\ \\
\LEM{lemageneral}(ii)   & $\displaystyle{\int^{+\infty}{\frac{s^{p-1}}{g(s)}}\diff s< +\infty}$ & $v'(r)\xrightarrow[r\rightarrow R]{}\infty$,
\vspace{0.2cm}\\ && $\scriptstyle{\mbox{and}}$\vspace{0.2cm}\\
&& $v(r)\leq M$ with $0<M<\infty$ \\ \\
\REM{weaklp} & $\displaystyle{\int^{+\infty}{\frac{s^{2(p-1)}}{g(s)}}\diff s= +\infty}$ & $v'(r)\notin L^{p}(0,R)$\\ \\
\end{tabular}
\end{rmk}

We now continue with the proof of the existence result in \THM{first}.
\begin{proof}[Proof of \THM{first} (ii)]
We want to prove that we have a positive solution in the whole $\rn$. For this purpose we are going to prove that $v$, the radial classical solution that we already proved to exist in a neighborhood of zero, is defined for all $r>0$. Suppose by contradiction that the solution exists in a finite maximal interval $(0,R)$, $R<\iy$ is finite.

First, we affirm that $v(r)\rightarrow \infty$ and $v'(r)\rightarrow\infty$ as $r\rightarrow R$.

In fact suppose that $v(r)$ is bounded as $r\rightarrow R$. By the continuity and monotonicity of $v$ and $f$
\begin{equation*}
((v')^{p-1})'\leq f(C)+g(v').
\end{equation*}
Dividing by $g(v')$ and integrating from $r_{0}$ to $r$ with $0<r_{0}<r<R$ we obtain
\begin{equation*}
(p-1)\int_{v'(r_{0})}^{v'(r)} \frac{s^{p-2}}{g(s)} \diff s\leq C(r-r_{0})
\end{equation*}
and letting $r\rightarrow R$
\begin{equation*}
(p-1)\int_{v'(r_{0})}^{\infty} \frac{s^{p-2}}{g(s)} \diff s\leq C(R-r_{0}),
\end{equation*}
which is a contradiction with our condition \EQ{kog}. The conclusion about $v'$ follows from \LEM{prad} in the same way as in the proof of \LEM{lemageneral}.

Now let's define
\begin{equation}
A(r)=\frac{r^{n-1}v'}{(F(v(r)))^\frac{1}{p}}
\end{equation}
and  separate the proof into several cases, according to the asymptotic behavior of $A(r)$ as $r\rightarrow R$.
\begin{itemize}
    \item []\textit{Case 1} Suppose that $A(r)$ is bounded when $r\rightarrow R$.
\end{itemize}
Then taking the integral between $R/2$ and any $r\in (R/2, R)$ we obtain
\begin{equation*}
\left(\frac{R}{2}\right)^{n-1}\int^{r}_{R/2}\frac{v'}{F(v(s))^{\frac{1}{p}}}\diff s \leq \int^{r}_{R/2}s^{n-1}\frac{v'}{F(v(s))^{\frac{1}{p}}}\diff s =\int^{r}_{R/2}A(s) \diff s\leq c (r-R/2)
\end{equation*}
Letting $r\rightarrow R$ we obtain that the term to the right is bounded. This is a contradiction with the hypothesis of the condition \EQ{kof} since
\begin{equation*}
\left(\frac{R}{2}\right)^{n-1}\int^{r}_{R/2}\frac{v'}{F(v(s))^{\frac{1}{p}}}\diff s= \left(\frac{R}{2}\right)^{n-1}\int^{v(r)}_{v(R/2)}\frac{1}{F(s)^{\frac{1}{p}}}\diff s.
\end{equation*}
\begin{itemize}
    \item []\textit{Case 2}
Suppose that $A(r)\rightarrow \infty$ when $r\rightarrow R$.
\end{itemize}
 Let $w=(v')^{p-1}$ and $H=F(v)^{\frac{1}{p}}$, thus the assumption of this case can be written as
\begin{equation}\label{eq:asscase2}
\frac{H}{w^{\frac{1}{p-1}}}\xrightarrow[r\rightarrow R]{} 0
\end{equation}
and the problem with the specific equation as in \EQ{edomp} can be now rewritten as
\begin{equation}\label{eq:pmdmud}
\begin{cases}
w'+\frac{n-1}{r}w&=p \frac{H^{p-1}H'}{w^{\frac{1}{p-1}}}+g\left(w^{\frac{1}{p-1}}\right)\\
\;\;\;\;\;\;\;\;\;\; w(0)&=0.
\end{cases}
\end{equation}
By the properties shown in \LEM{prad} we have that $w\geq 0$ which implies that
\begin{equation}\label{eq:pconta}
w'\leq p \frac{H^{p-1}H'}{w^{\frac{1}{p-1}}}+g\left(w^{\frac{1}{p-1}}\right).
\end{equation}
On the other hand we can fix $r_{0}>0$ such that $w\geq 1$ in $(r_{0},R)$, by the convergence to infinity of $w$ when $r$ tends to $R$.
Let $S=S(t)$ be the solution  of the initial value problem
\begin{equation*}
\begin{cases}
S'(t)&=(p-1)g\left(S^{\frac{1}{p-1}}(t)\right)\\
\; S(1)&=1,
\end{cases}
\end{equation*}
defined implicitly by
\begin{equation*}
t=1+\int^{S(t)}_{1}\frac{\sigma^{p-2}}{g(\sigma)}\diff \sigma \;\;\; t\in [1,\infty).
\end{equation*}
From the hypothesis of $g$ and the condition \EQ{kog} we deduce that $S$ is bijective from $[1,\infty)$ to $[1,\infty)$. Then by taking the inverse of this function we get
\begin{equation*}
\left(S^{-1}\left(w \right)\right)'=\frac{w'}{S'\left(S^{-1}(w)\right)}=\frac{w'}{(p-1)g(w^{\frac{1}{p-1}})}.
\end{equation*}
Then, by \EQ{pconta} we have
\begin{equation}\label{eq:casi}
\left(S^{-1}\left(w\right)\right)'\leq \frac{1}{p-1}+\frac{p}{p-1}\frac{H^{p-1} H'}{w^{\frac{1}{p-1}}g(w^{\frac{1}{p-1}})}.
\end{equation}
On the other hand,
\begin{equation*}
\left(S^{-1}\left(H^{p-1}\right)\right)'=\frac{H'H^{p-2}}{g\left(H\right)}.
\end{equation*}
By substituting this in \EQ{casi} we get
\begin{eqnarray}\label{eq:esta}
\left(S^{-1}\left(w\right)\right)'&\leq& \frac{1}{p-1}+\frac{p}{p-1}\frac{H^{p-1}}{w^{\frac{1}{p-1}}}\frac{g\left(H\right)}{g(w^{\frac{1}{p-1}}) H^{p-2}} \left(S^{-1}\left(H\right)\right)'\\
&=&\frac{1}{p-1}+\frac{p}{p-1}\frac{H}{w^{\frac{1}{p-1}}}\frac{g\left(H\right)}{g(w^{\frac{1}{p-1}})} \left(S^{-1}\left(H\right)\right)'.
\end{eqnarray}
We also know by the monotonicity of $f$ that $H$ is increasing; hence we can choose $r_{1}$ such that $H\geq 1$ in $(r_{1},R)$.

Since \EQ{asscase2} is satisfied and  $g$ is monotone, we can say that exist $r_{2}>r_{1}$ such that
\begin{equation*}
\frac{p}{p-1}\frac{H}{w^{\frac{1}{p-1}}}\frac{g\left(H\right)}{g(w^{\frac{1}{p-1}})}\leq \frac{1}{p}
\end{equation*}
in $(r_{2},R)$. Then
\begin{equation*}
(S^{-1}(w))'\leq 1+\frac{1}{p} (S^{-1}(H))' \;\; \mbox{in} \;\; (r_{2}, R).
\end{equation*}
Integrating from $r_{1}$ to $r \in(r_{2}, R)$
\begin{eqnarray*}
(S^{-1}(w))'&\leq & S^{-1}(w(r_{1}))+\frac{1}{p-1}(R-r_{1})+\frac{1}{p} S^{-1}(H^{p-1}(r))\\
&-&\frac{1}{p} S^{-1}(H^{p-1}(r_{2})) +\int_{r_{1}}^{r_{2}}\frac{p}{p-1}\frac{H}{w^{\frac{1}{p-1}}}\frac{g(H)}{g(w^{\frac{1}{p-1}})}S^{-1}(H^{p-1})\\
&=& C(r_{1}, r_{2},p,R) +\frac{1}{p} S^{-1}(H^{p-1}(r))
\end{eqnarray*}
where $C(r_{1}, r_{2},p,R)$ is a constant independent of $r$.

Since $F$ is increasing we have that $H(r)\rightarrow \infty$ as $r\rightarrow R$. Hence by the definition of $S$ and the condition \EQ{kog} $S^{-1}(H(r))\rightarrow \infty$ as $r\rightarrow R$. Then we find an $r_{3}\in(r_{1},R)$ such that
\begin{equation*}
S^{-1}(w)<S^{-1}(H^{p-1}) \;\; \mbox{in} \;\; (r_{3},R).
\end{equation*}
Since $S^{-1}$ is increasing this means that $w< H$ in $(r_{3},R)$. But this is a contradiction with the fact of $\frac{H}{w^{\frac{1}{p-1}}}\rightarrow 0$ as $r\rightarrow R$.
\begin{itemize}
    \item []\textit{Case 3} Assume that we are in none of the above cases, i.e. $A(r)$ is neither bounded nor tends to infinity as $r\rightarrow R, \; r<R$.
\end{itemize}
 This means that
\begin{equation*}
\limsup_{r\rightarrow R} A(r)=+\infty \;\; \mbox{and} \;\; A_{0}:=\liminf_{r\rightarrow R} A(r)<+\infty .
\end{equation*}
Then for each $\hat{A}>A_{0}$ we can find sequences $s_{n}, t_{n}\rightarrow R$ such that
\begin{equation*}
A(s_{n})=A(t_{n})=\hat{A} \;\;\;\; A'(s_{n})\geq 0 \;\;\;\; A'(t_{n})\leq 0.
\end{equation*}
The derivative of $A$ is
\begin{eqnarray}
\label{eq:deriv}
A'(r)&=&\frac{(r^{n-1}v')' F^{\frac{1}{p}}(v)-(r^{n-1}v')(F^{\frac{1}{p}}(v))'}{F^{\frac{2}{p}}(v)}\\
     &=&\frac{(r^{n-1}v')'}{F^{\frac{1}{p}}(v)}-\frac{r^{n-1}v'}{F^{\frac{2}{p}}(v)}\left(\frac{1}{p}F^{\frac{1}{p}-1}(v) f(v)v'\right)\nonumber\\
     &=&\frac{(r^{n-1}v')'}{F^{\frac{1}{p}}(v)}-\frac{r^{n-1}(v')^{2}f(v)}{p F^{\frac{1}{p}+1}(v)}.\nonumber
\end{eqnarray}
Since $v'\geq 0$ multiplying \EQ{deriv} by $(v')^{p-2}$
\begin{equation*}
A'(r)(v')^{p-2}=\frac{(r^{n-1}v')'(v')^{p-2}}{F^{\frac{1}{p}}(v)}-\frac{r^{n-1}(v')^{p}f(v)}{p F^{\frac{1}{p}+1}(v)}.
\end{equation*}
Adding and subtracting the factor $(p-2)r^{n-1}\frac{(v')^{p-2}v''}{F^{\frac{1}{p}}}$ on the right hand side gives
\begin{align*}
& A'(r)(v')^{p-2}\\
=&\frac{(r^{n-1}v')'(v')^{p-2}}{F^{\frac{1}{p}}(v)}+(p-2)r^{n-1}\frac{(v')^{p-2}v''}{F^{\frac{1}{p}}}-(p-2)r^{n-1}\frac{(v')^{p-2}v''}{F^{\frac{1}{p}}}-\frac{r^{n-1}(v')^{p}f(v)}{p F^{\frac{1}{p}+1}(v)}\\
=&\frac{(r^{n-1})'(v')^{p-1}+ r^{n-1} (v')^{p-2}v''}{F^{\frac{1}{p}}(v)}+(p-2)r^{n-1}\frac{(v')^{p-2}v''}{F^{\frac{1}{p}}}\\
&\hspace{4.8cm}-(p-2)r^{n-1}\frac{(v')^{p-2}v''}{F^{\frac{1}{p}}}-\frac{r^{n-1}(v')^{p}f(v)}{p F^{\frac{1}{p}+1}(v)}\\
=&\frac{(r^{n-1})'(v')^{p-1}+(p-1) r^{n-1} (v')^{p-2}v''}{F^{\frac{1}{p}}(v)}-(p-2)r^{n-1}\frac{(v')^{p-2}v''}{F^{\frac{1}{p}}}-\frac{r^{n-1}(v')^{p}f(v)}{p F^{\frac{1}{p}+1}(v)}\\
=&\frac{(r^{n-1})'(v')^{p-1}+r^{n-1}((v')^{p-1})'}{F^{\frac{1}{p}}(v)}-(p-2)r^{n-1}\frac{(v')^{p-2}v''}{F^{\frac{1}{p}}}-\frac{r^{n-1}(v')^{p}f(v)}{p F^{\frac{1}{p}+1}(v)}\\
=&\frac{(r^{n-1}(v')^{p-1})'}{F^{\frac{1}{p}}(v)}-(p-2)r^{n-1}\frac{(v')^{p-2}v''}{F^{\frac{1}{p}}}-\frac{r^{n-1}(v')^{p}f(v)}{p F^{\frac{1}{p}+1}(v)}\\
=&\frac{r^{n-1}(f(v)+g(v'))}{F^{\frac{1}{p}}(v)}-(p-2)r^{n-1}\frac{(v')^{p-2}v''}{F^{\frac{1}{p}}}-\frac{r^{n-1}(v')^{p}f(v)}{p F^{\frac{1}{p}+1}(v)}\\
=&\frac{r^{n-1}f(v)}{F^{\frac{1}{p}}(v)}\left(1+\frac{g(v')}{f(v)}-\frac{1}{p}\frac{(v')^{p}}{F(v)}\right)-(p-2)r^{n-1}\frac{(v')^{p-2}v''}{F^{\frac{1}{p}}}.
\end{align*}
Let $W(r):=1+\frac{g(v')}{f(v)}-\frac{1}{p}\frac{(v')^{p}}{F(v)}$, then,
\begin{equation*}
A'(r)(v')^{p-2}= \frac{r^{n-1}f(v)}{F^{\frac{1}{p}}(v)}W(r) -(p-2)r^{n-1}\frac{(v')^{p-2}v''}{F^{\frac{1}{p}}}.
\end{equation*}
Now we deal with the cases $p\leq 2$ and $p\geq 2$ separately.
\begin{itemize}
    \item [] Case 1 \;\;\; $p\leq 2$
\end{itemize}
When $p\leq 2$ it is clear that $(p-2)r^{n-1}\frac{(v')^{p-2}v''}{F^{\frac{1}{p}}}\leq 0$, thus
\begin{equation*}
A'(r)(v')^{p-2}\geq \frac{r^{n-1}f(v)}{F^{\frac{1}{p}}(v)}W(r).
\end{equation*}
 Set $\bar{A}(r)=\frac{A(r)}{r^{n-1}}$. We have
\begin{eqnarray*}
W(r)&=& 1+\frac{g(\frac{A(r)}{r^{n-1}}F^{\frac{1}{p}})}{f(v)}-\frac{1}{p}\frac{A^{p}(r)}{r^{(n-1)p}}\\
&=& 1+\frac{g(\bar{A}(r)F^{\frac{1}{p}})}{f(v)}-\frac{1}{p}\bar{A}^{p}\\
&=& 1+ \bar{A}^{p}(r) \left(\frac{1}{\bar{A}^{p}(r)}\frac{g(\bar{A}(r)F^{\frac{1}{p}})}{f(v)}-\frac{1}{p}\right).
\end{eqnarray*}
Note that $\bar{A}(t_{n})\rightarrow \tilde{A}:=\hat{A} R^{1-N}$, then for each $\epsilon>0$ there exists $n_{0}\in \N$ such that for all $n\geq n_{0}$ we have $\bar{A}(t_{n})\in (\tilde{A}-\epsilon, \tilde{A}+\epsilon)$. Then by the monotonicity of $g$
\begin{equation}
W(t_{n})\geq 1+\left(\bar{A}(t_{n})\right)^{p}\left(\rho(\epsilon)^{p}\frac{g\left((\tilde{A}(r)-\epsilon)(F(v))^{\frac{1}{p}}\right)}{\left(\tilde{A}-\epsilon\right)^{p} f(v)}-\frac{1}{p}\right),
\end{equation}
with $\rho(\epsilon):=\frac{\tilde{A}-\epsilon}{\tilde{A}+\epsilon}\rightarrow 1$ as $\epsilon\rightarrow 0$.
Fix $\epsilon>0$ sufficiently small and $\hat{A}$ large that \EQ{ass1} be satisfied; then  $W(t_{n})>0$ for $n$ large. Consequently $A'(t_{n})>0$, a contradiction.
\begin{itemize}
    \item [] Case 2\;\;\;  $p> 2$
\end{itemize}
If $p\geq 2$ then $(p-2)r^{n-1}\frac{(v')^{p-2}v''}{F^{\frac{1}{p}}}\geq 0$. Thus,
\begin{equation*}
A'(r)(v')^{p-2}\leq \frac{r^{n-1}f(v)}{F^{\frac{1}{p}}(v)}W(r).
\end{equation*}
Then, analogously to the other case, since $\bar{A}(s_{n})\rightarrow \tilde{A}:=\hat{A} R^{1-N}$ for each $\epsilon>0$ there exists $n_{0}\in \N$ such that for all $n\geq n_{0}$ we have $\bar{A}(s_{n})\in (\tilde{A}-\epsilon, \tilde{A}+\epsilon)$. Using the monotonicity of $g$
\begin{equation*}
W(s_{n})\leq 1+\left(\bar{A}(s_{n})\right)^{p}\left(\rho(\epsilon)^{p}\frac{g\left((\tilde{A}(r)+\epsilon)(F(v))^{\frac{1}{p}}\right)}{\left(\tilde{A}+\epsilon\right)^{p} f(v)}-\frac{1}{p}\right).
\end{equation*}
Fixing $\epsilon>0$ sufficiently small and $\hat{A}$ large that \EQ{ass2} be satisfied  $W(s_{n})<0$ . This is a contradiction with the established sign for $A'(s_{n})$.
This means that the third case for $A$ is impossible too. Therefore the radial classical solution is defined in whole $\rn$.
\end{proof}
\begin{rmk}
Note that the limit conditions in \EQ{ass1} and \EQ{ass2} are different according to the value of $p$. Thus, for example, from this theorem we can not conclude about the existence for the standard functions $f(t)=g(t)=t^{q}$ where $q\leq p-1$ if $p<2$, as we mentioned earlier.

However, as we observed in \REM{altcond}, the conditions \EQ{ass1} and \EQ{ass2} can be removed in the particular case when $g$ has  growth $s^{p-1}$ at most, i.e
\be\label{eq:growc}
\limsup_{s\rightarrow \infty} \frac{g(s)}{s^{p-1}}<+\infty .
\ee
In fact, by \EQ{pmdmud} we have
\bet
w'w^{\frac{1}{p-1}}- g(w^{\frac{1}{p-1}})w^{\frac{1}{p-1}}\leq (H^{p})'.
\eet
Now using the growth condition \EQ{growc} we obtain that
\bet
w'w^{\frac{1}{p-1}}- c w^{\frac{p}{p-1}}\leq (H^{p})'.
\eet
Letting $\widetilde{w}= w^{\frac{p}{p-1}}$,\; $\widetilde{H} = H^{p}$ and $c_{1}= c\frac{p}{p-1}$ the above inequality can be rewritten as
\bet
\widetilde{w}'-c c_{1}\widetilde{w}\leq c_{1}\widetilde{H}'.
\eet
Consequently,
\begeqnat
\widetilde{w}&\leq& e^{cc_{1}}\int_{r_{0}}^{r}e^{-cc_{1}s}c_{1}\widetilde{H}'(s)\diff s\\
&\leq &c_{1}R e^{cc_{1}}\widetilde{H}(s)
\eeqnat
for some $0<r_{0}<r$, which implies that $\ds{\frac{\widetilde{w}}{\widetilde{H}}\leq C(R)}$, that is $\ds{\frac{w^{\frac{1}{p-1}}}{H}\leq C(p,R)}$. Thus we only have the first case in the above proof, in which no hypothesis other than the (KO) integral conditions is used.
\end{rmk}

To focus on the important case of the power nonlinearity $g(t)=t^{q}$, $q>0$, we have the following result.
\begin{cor}
Let $f$ that satisfy \EQ{fgcond}. If $q>p-1$ the only non-negative weak solution in $W_{\loc}^{1,\infty}(\rn)$ of
\begin{equation}
\label{eq:partpm}
\plap u= f(u)+\left|\nabla u\right|^{q} \;\;\; \mbox{in} \;\; \rn
\end{equation}
is the trivial one.
\end{cor}

\section{Existence and nonexistence of entire solutions for ($P_{-}$)}\label{sec:entsoln}
In this section we study the solutions, in the smooth sense,  of $(P_{-})$
\be\label{eq:Pmenos}
\plap u = f(u)-g(\left|\nabla u\right|) .
\ee
To prove the main result, \THM{second} (i), first we prove the following lemma.
\begin{lem}\label{lem:pmaedot}
There exists an increasing strict super-solution $\bar{v}$ of
\begin{equation}\label{eq:pmaedot}
\begin{cases}
((v')^{p-1})'+\frac{n-1}{r}(v')^{p-1}&\leq f(v)-g(v')\\
\;\;\;\;\;\;\;\;\;\;\;\;\; \;\;\;\;\;\;\;\;\;\;\;\;\;\;\;\;\;\; v(0)&=\bar{v}_{0}\\
\;\;\;\;\;\;\;\;\;\;\;\;\; \;\;\;\;\;\;\;\;\;\;\;\;\;\;\;\;\;\; v'(0)&=0.
\end{cases}
\end{equation}
which ceases to exist at a finite $R$, and satisfies $\bar{v}(r)\rightarrow \infty$ when $r\rightarrow R$, for $\bar{v}_{0}$ large enough.
\end{lem}
\begin{proof}[Proof]
Suppose that $R\leq (\frac{p-1}{p})^{p-1}$, we search for a super-solution in the form
\begin{equation*}
\bar{v}(r)=\phi (R^{\frac{p}{p-1}}-r^{\frac{p}{p-1}})
\end{equation*}
where $\phi$ is a decreasing function that is defined implicitly by
\begin{equation*}
\int_{\phi(t)}^{\infty} \frac{\diff s}{\Gamma^{-1}(F(s))}=t,
\end{equation*}
which is a super-solution of \EQ{pmaedot} if
\begin{equation}\label{eq:supersf}
(\frac{p}{p-1})^{p}(p-1)r^{\frac{p}{p-1}}\left|\phi '\right|^{p-2}\phi''+(\frac{p}{p-1})^{p-1}n\left|\phi '\right|^{p-1}\leq f(\phi)-g(\frac{p}{p-1}r^{\frac{1}{p-1}}\left|\phi '\right|)
\end{equation}
where the comma means the derivative with respect to $t=R^{\frac{p}{p-1}}-r^{\frac{p}{p-1}}$.
Using the conditions \EQ{fgcond} and by the hypothesis \EQ{pmecond1} $\phi$ is well defined and it can be verified that $\phi(t)\rightarrow \infty$ as $t\rightarrow 0$, $t>0$; $\phi'(t)<0$ and $\Gamma(\left|\phi'\right|)= F(\phi)$.
Then we can use the inequation
\begin{equation}\label{eq:supcond}
\left|\phi '\right|^{p-2}\phi''+(\frac{p}{p-1})^{p-1}n\left|\phi '\right|^{p-1}+g(\left|\phi '\right|)\leq f(\phi)
\end{equation}
instead of \EQ{supersf} to verify the character of super-solution of $\phi$. Observe that
\begin{equation}\label{eq:eqint}
\frac{\Gamma^{-1}(F(s))}{s}\rightarrow \infty \; \; \;  \mbox{as} \; \; \; s\rightarrow \infty,
\end{equation}
hence
\begin{equation*}
\frac{\left|\phi '(t)\right|}{\phi(t)}\rightarrow \infty \; \; \;  \mbox{as} \; \; \; t\rightarrow 0.
\end{equation*}
Thus, we can affirm that exists $\epsilon>0$ such that $\phi(t)\leq \frac{1}{2}\left|\phi '(t)\right|$ if $t\in (0,\epsilon)$. Now we take $R$ such that $t\in (0,\epsilon)$ for $r\in (0, R)$.

Finally we check that $\phi$ satisfies \EQ{supcond}, so is the desired function. By differentiating with respect to $t$ in \EQ{eqint} we obtain
\begin{equation}
\phi''=\frac{f(\phi)\left|\phi '\right|}{2 g(2\left|\phi '\right|)+(p-1)c(p,N)\left|\phi '\right|^{p-1}}.
\end{equation}
We can see that $\phi''>0$ and
\begin{equation}\label{eq:eqint2}
\left|\phi '\right|^{(p-2)}\phi''\leq\frac{f(\phi)}{(p-1)c(p,N)}.
\end{equation}
On the other hand we have
\begin{equation*}
F(t)=\int_{0}^{t} f(s)\diff s\leq f(t)t
\end{equation*}
and
\begin{equation*}
\Gamma(t)=\int_{0}^{2t} g(s)\diff s+\frac{p-1}{p}c(p,N)t^{p}\geq tg(t)+\frac{p-1}{p}c(p,N)t^{p}
\end{equation*}
thus, by taking $c(p,N)=(\frac{p}{p-1})^{p}n$ we obtain
\begin{equation}\label{eq:eqint3}
g(\left|\phi '\right|)+(\frac{p}{p-1})^{p-1}n\left|\phi '\right|^{p-1}\leq \frac{\Gamma(\left|\phi '\right|)}{\left|\phi '\right|}=\frac{F(\phi)}{\left|\phi '\right|}\leq f(\phi)\frac{\phi}{\left|\phi '\right|}\leq \frac{1}{2}f(\phi).
\end{equation}
By adding \EQ{eqint2} and \EQ{eqint3} we get the expression \EQ{supcond} and  conclude.
\end{proof}
Using this result we can prove the \THM{second}.
\begin{proof}[Proof of \THM{second} (i)]
 Suppose by contradiction that we have  a non trivial solution $w$ of $(P_{-})$, defined in the whole $\rn$. Let  $\tilde{v}$ be a solution of \EQ{pmaedot}  with $\tilde{v}_{0}=\frac{w(0)}{2}$ -- we know such a solution exists by Lemma \ref{lem:solmax}.

Now we show that $\tilde{v}$ is defined in the whole $\rn$.
By \LEM{prad} we know that  $\tilde{v}'>0$ and $((\tilde{v}')^{p-1})'\geq 0$. Then
\bet
\tilde{v}'\leq \frac{r}{n-1} f (\tilde{v}).
\eet
Thus if $\tilde{v}$ exists in some maximal interval $(0,\bar{R})$ with $\bar{R}<\infty$ and becomes infinite at  $\bar{R}$,  we have
\begin{equation*}
w(x)<\tilde{v}(\left|x\right|) \;\;\; \mbox{in}\;\; \partial B_{r_{0}}(0),
\end{equation*}
for $0<r_{0}<\bar{R}$ sufficiently near to $\bar{R}$. Using the comparison principle in $B_{r_{0}}(0)$ we obtain that
\begin{equation*}
w(x)<\tilde{v}(\left|x\right|) \;\;\; \mbox{in}\;\;   B_{r_{0}}(0).
\end{equation*}
This is a contradiction the initial hypothesis that $\tilde{v}(0)=\tilde{v}_{0}<w(0)$.

 Therefore $R=\infty$ and $\tilde{v}(r)\rightarrow \infty$ as $r\rightarrow \infty$, since $\tilde {v}$ is increasing and $((\tilde{v'})^{p-1})'\geq 0$. By using again the comparison principle we have
\begin{equation*}
w(x)\not\leq \tilde{v}(\left|x\right|) \;\;\; \mbox{in}\;\; \partial B_{r}(0),
\end{equation*}
for all  $r\in (0,\infty)$. This implies that exist a sequence $x_{n}\in \R^{n}$ with $\left|x_{n}\right|\rightarrow \infty$ such that $w(x_{n})\rightarrow\infty$. Fixing $n_{0}$ large enough we obtain $w(x_{n_{0}})>\bar{v}_{0}$, where $\bar{v}_{0}$ is the number obtained in Lemma \ref{lem:pmaedot}.

Repeating the above for $\bar{v}(\left|x-x_{n_{0}}\right|)$ instead of $\tilde{v}(\left|x\right|)$, with $\bar{v}$ the function of Lemma~\ref{lem:pmaedot}, we obtain a contradiction.
\end{proof}

\begin{rmk}
The integral condition \EQ{pmecond1} is sharp -- we can compute that
\begin{equation*}
\divg \left(\left|\nabla u\right|^{p-2}\nabla u\right)\geq(p-1)\left|u\right|^{p-2}u-\left|\nabla u\right|^{p}
\end{equation*}
has the non-negative nontrivial solution $u=\exp x_{1}$.
\end{rmk}
\begin{proof}[Proof of \THM{second} (ii)]
Let $v$ be a  solution of
\begin{equation}\label{eq:pmaedotig}
\begin{cases}
((v')^{p-1})'+\frac{n-1}{r}(v')^{p-1}&= f(v)-g(v')\\
\;\;\;\;\;\;\;\;\;\;\;\;\; \;\;\;\;\;\;\;\;\; \;\;\;\;\;\;\;\; v(0)&=v_{0}>0\\
\;\;\;\;\;\;\;\;\;\;\;\;\; \;\;\;\;\;\;\;\;\; \;\;\;\;\;\;\; v'(0)&=0
\end{cases}
\end{equation}
in a maximal interval $(0,R)$, \; $0<R<\infty$.
By the  results in \LEM{prad} we have that $f(v)-g(v')\geq 0$, hence
\begin{equation}\label{eq:imp1}
g^{-1}(f(v))\geq v'.
\end{equation}
Also by the properties of $v$ we know that if $v$ is defined in a maximal interval $(0,R)$ with $R<\infty$, $v(r)\rightarrow\infty$ as $r\rightarrow R$ and since $v$ satisfy \EQ{pmaedotig}
\begin{equation*}
((v')^{p-1})'\leq f(v)
\end{equation*}
from which we get
\begin{equation}\label{eq:imp2}
v'\leq (\frac{p}{p-1})^{\frac{1}{p}}(F)^{\frac{1}{p}}.
\end{equation}
By integrating \EQ{imp1} and \EQ{imp2} we obtain that
\begin{equation*}
\max\left\{\int_{v_{0}}^{v(r)}\frac{\diff s}{(F(s))^{\frac{1}{p}}}, \int_{v_{0}}^{v(r)}\frac{\diff s}{g^{-1}(f(s))}\right\}\leq (\frac{p}{p-1})^{\frac{1}{p}}r, \;\;\; r\in (0,R).
\end{equation*}
Letting $r\rightarrow R$ we get a contradiction with the assumption of the theorem. Then a solution of the problem is $v(x)=u(\left|x\right|)$ for all $x\in \rn$.
\end{proof}
\begin{rmk}
To obtain \EQ{imp1} we only use that the radial expression of the operator is nonnegative. Thus if we know that $f(v)-g(\vl)\geq 0$ regardless of the operator we obtain the integral condition $\ds{\int_{1}^{\iy}\frac{1}{g^{-1}(f(s))}\diff s}=\iy$ as a sufficient condition for the existence of solution.
\end{rmk}
Analyzing the particular case of the nonlinearity $g(t)=t^{q},\; q>0$ we have $g^{-1}(t)=t^{\frac{1}{q}}$ and $\Gamma (s)=\frac{(2s)^{q+1}}{q+1}+(\frac{p}{p-1})^{p-1}n s^{p}$. Then
\begin{itemize}
\item If $0<q\leq p-1$  \;\; $\Gamma(s)\sim \mbox{cte}\; s^{p}$ when $s\rightarrow \infty$. Thus $\Gamma^{-1}(s)\sim \mbox{cte}\; s^{\frac{1}{p}}$ when $s\rightarrow \infty$.
\item If $q> p-1$  \;\; $\Gamma(s)\sim \mbox{cte}\; s^{q+1}$ when $s\rightarrow \infty$. Thus $\Gamma^{-1}(s)\sim \mbox{cte}\; s^{\frac{1}{q+1}}$ when $s\rightarrow \infty$.
\end{itemize}
This implies the following corollary.
\begin{cor}
Let $f$ be a continuous function satisfy \EQ{fgcond}. Then
\begin{itemize}
\item [(i)] If $0<q\leq p-1$ the equation
\begin{equation}
\label{eq:cppmen}
\plap u =f(u)-\left|\nabla u\right|^{q}
\end{equation}
admits a positive solution if and only if the condition \EQ{kof} is satisfied.
\item [(ii)] If $q>p-1$ the equation \EQ{cppmen} has at least one positive solution if
\begin{equation*}
\int_{1}^{\infty}\frac{1}{f(s)^{\frac{1}{q}}}=\infty
\end{equation*}
 and any solution vanished identically if
\begin{equation*}
\int_{1}^{\infty}\frac{1}{F(s)^{\frac{1}{q+1}}}<\infty.
\end{equation*}
\end{itemize}
\end{cor}
By taking $f(t)=t^{m},\; m>0$ in the above, the first affirmation  is satisfied if and only if $m<q$ with $q>p-1$. On another hand if $q\leq p-1$ there is a positive solution if and only if $m>q$.


\bibliographystyle{siam}
\bibliography{references}

\end{document}